\documentclass[11pt,english]{smfart}

\usepackage{amsmath,leftidx,amsthm}
\usepackage[all]{xy}
\usepackage{xspace,smfthm}
\usepackage[psamsfonts]{amssymb}
\usepackage[latin1]{inputenc}
\usepackage{graphicx,color,mathrsfs}
\usepackage{hyperref,fancyhdr, dutchcal,appendix}

\newtheorem*{theor}{\bf Main Theorem}

\newtheorem{theorem}{\bf Theorem}
\newtheorem{proposition}[theorem]{\bf Proposition}

\newtheorem{lemma}[theorem]{\bf Lemma}

\def\C{{\mathbb C}}

\def\R{{\mathbb R}}

\def\Q{{\mathbb Q}}

\def\p{\mathbb{P}}

\title[Variation of canonical heights]{Variation of canonical heights of subvarieties for polarized endomorphisms}
\author{Thomas Gauthier}
\address{Laboratoire de Math\'ematiques d'Orsay, B\^atiment 307, Universit\'e Paris-Saclay, 91405 Orsay Cedex, France}
\email{thomas.gauthier1@universite-paris-saclay.fr}
\author{Gabriel Vigny}
\address{LAMFA, Universit\'e de Picardie Jules Verne, 33 rue Saint-Leu, 80039 AMIENS Cedex 1, FRANCE}
\email{gabriel.vigny@u-picardie.fr}
\thanks{The first author is partially supported by the Institut Universitaire de France. The second author is partially supported by the ANR QuaSiDy /ANR-21-
	CE40-0016.}

\begin{document}

\begin{abstract}
When an endomorphism $f:X\to X$ of a projective variety which is polarized by an ample line bundle $L$, i.e. such that $f^*L\simeq L^{\otimes d}$ with $d\geq2$, is defined over a number field, Call and Silverman defined a canonical height $\widehat{h}_f$ for $f$.
In a family $(\mathcal{X},\mathcal{f},\mathcal{L})$ parametrized by a curve $S$ together with a section $P:S\to \mathcal{X}$, they show that $\widehat{h}_{f_t}(P(t))/h(t)$ converges to the height $\widehat{h}_{f_\eta}(P_\eta)$ on the generic fiber.

In the present paper, we prove the equivalent statement when studying the variation of canonical heights of subvarieties $Y_t$ varying in a family $\mathcal{Y}$ of any relative dimension. 
\end{abstract}

\maketitle

%
%


\section{Introduction}
A family $(\mathcal{X},\mathcal{f},\mathcal{L})$ of polarized endomorphisms parametrized by a smooth projective curve $S$ over a field $\mathbf{k}$ of characteristic $0$ is a family $\pi:\mathcal{X}\to S$ of projective $k$-varieties which is normal and flat over a Zariski open subset $S^0$ of $S$, a rational map $\mathcal{f}:\mathcal{X}\dashrightarrow\mathcal{X}$ which is regular over $S^0$ and a relatively ample line bundle $\mathcal{L}$ on $\mathcal{X}$, such that for each $t\in S^0$,  if $X_t:=\pi^{-1}\{t\}$ is the fiber of $\pi$ over $t$, $L_t:=\mathcal{L}|_{X_t}$ and $f_t:=\mathcal{f}|_{X_t}$, then $(X_t,f_t,L_t)$ is a polarized endomorphism, i.e. there is an integer $d\geq2$ such that $f_t^*L_t\simeq L_t^{\otimes d}$.
When $S$ and $(\mathcal{X},\mathcal{f},\mathcal{L})$ are defined over a number field $\mathbb{K}$, given a parameter $t\in S^0(\bar{\mathbb{Q}})$, one want to relate the arithmetic complexity of $t$, the dynamical complexity of the corresponding map $f_t$ and the dynamical complexity of the family $\mathcal{f}$. This can be done using the theory of heights.

\medskip

 For a polarized endomorphism $(X,f,L)$ defined over a product formula field $\mathbf{K}$, let $h_{X,L}$ be the standard Weil height function on $X(\bar{\mathbf{K}})$, relative to $L$. Call and Silverman \cite{CS-height} defined the \emph{canonical height} $\widehat{h}_f:X(\bar{\mathbf{K}})\to\R_+$ of the endomorphism $f$ as
 \[\widehat{h}_f=\lim_{n\to\infty}\frac{1}{d^n}h_{X,L}\circ f^n.\]
 Assume that $X$ is defined over the function field of characteristic zero  $\mathbf{K}:=\mathbb{K}(S)$ where $\mathbb{K}$ is a number field  and $S$ is a smooth projective $\mathbb{K}$-curve. To the polarized endomorphism $(X,f,L)$ we associate a \emph{model} $(\mathcal{X},\mathcal{f},\mathcal{L})$ over $S$, i.e. a family of polarized endomorphisms $(\mathcal{X},\mathcal{f},\mathcal{L})$ parametrized by $S$ such that, if $\eta$ is the generic point of $S$, then $(X,f,L)$ is isomorphic to $(X_\eta,f_\eta,L_\eta)$ where $X_\eta$ is the generic fiber of $\pi:\mathcal{X}\to S$, $f_\eta:=\mathcal{f}|_{X_\eta}$ and $L_\eta:=\mathcal{L}|_{X_\eta}$.

\medskip

 Endow $S$ with an ample $\mathbb{Q}$-line bundle and take $P\in X(\mathbf{K})$, $P$ can be thought of as a function $S\to \mathcal{X}$. In that setting,  we have the canonical height $\widehat{h}_{f_\eta}(P_\eta)$ which describes the arithmetic complexity of the orbit $\mathrm{Orb}_{f_\eta}(P)=(f_\eta^n(P_\eta))_n$ over $\mathbb{K}(S)$ and, given a parameter $t \in \bar{\mathbb{Q}}$, the naive height $h_S(t)$ which describes the arithmetic complexity of $t$, and the canonical height $\widehat{h}_{f_t}(P(t))$ which describes the arithmetic complexity of the orbit $\mathrm{Orb}_{f_t}(P(t))=(f_t^n(P(t))_n$ over $\bar{\mathbb{Q}}$. In that setting, Call and Silverman~\cite[Theorem~4.1]{CS-height} proved : 
 \begin{align}
\lim_{h_S(t)\to\infty \atop t\in S^0(\bar{\mathbb{Q}})}\frac{\widehat{h}_{f_t}(P(t))}{h_S(t)}=\widehat{h}_{f_\eta}(P_\eta).\label{ineg:CS-local2}
 \end{align}
 In the particular case where $X=\p^1(\mathbf{K})$ and $f$ is a polynomial map, Ingram \cite{Ingram_variation} improved \eqref{ineg:CS-local2} by showing there is an effective $\mathbb{Q}$-divisor $D(f,P)$ on $S$ of degree $\widehat{h}_{f_\eta}(P_\eta)$ such that $\widehat{h}_{f_t}(P(t)) = h_{D(f,P)}(t)+O_{f,P,S}(1)$ (see also Tate \cite{Tate_elliptic} for the case of families of elliptic curves) and finally the first author and Favre showed in \cite{MR3864204} that the height function $\widehat{h}_{f_t}(P(t))$ is induced by a continuous adelic metrization of the line bundle $\mathcal{O}(D(f,P))$. Very recently, Ingram also improved \eqref{ineg:CS-local2} in \cite{ingram2021variation} saving a power in the error term. 
 
 \medskip

 Nevertheless, when the relative dimension of $\mathcal{X}$ is at least $2$, it can be useful to consider the canonical height of fibers of a subvariety $\mathcal{Y}\subsetneq \mathcal{X}$ with $\pi(\mathcal{Y})=S$ of positive relative dimension.  Indeed, generalizing the $1$-dimensional theory \cite{MSS, Lyubich-unstable}, Berteloot, Bianchi and Dupont \cite{BBD} showed that bifurcations in a complex family of endomorphisms of the projective space $\mathbb{P}^k$ are caused by the unstability of the critical set (which has codimension $1$), and the authors of op. cit., following DeMarco~\cite{DeMarco1} in dimension 1, defined a \emph{bifurcation current} which gives a measurable meaning to bifurcations. The authors showed in \cite{GV_Northcott} that, in the case of an algebraic family of endomorphisms of the projective space $\mathbb{P}^k$, the mass of this current is actually the canonical height of the critical divisor.   
 
 ~
 
Here is the main result of this article.

\begin{theor}\label{tm:principal1}
Let $(\mathcal{X},\mathcal{f},\mathcal{L})$ be a family of polarized endomorphisms over $S$ and let $\mathcal{Y}\subsetneq\mathcal{X}$ be an irreducible subvariety such that $\pi(\mathcal{Y})=S$, all defined over a number field $\mathbb{K}$. For any $\mathbb{Q}$-ample height $h_S$ on $S$ of degree $1$, we have
\[\lim_{h_S(t)\to\infty \atop t\in S_\mathcal{Y}^0(\bar{\mathbb{Q}})}\frac{\widehat{h}_{f_t}(Y_t)}{h_S(t)}=\widehat{h}_{f_\eta}(Y_\eta),\]
where $S^0_\mathcal{Y}$ be the maximal Zariski open subset of $S^0$ over which $\pi|_\mathcal{Y}$ is flat and projective.
\end{theor}

Ingram~\cite{Ingram-variation} proved this result when $\mathcal{Y}=\mathrm{Crit}(f)$ is the critical locus of the family $f$ using a different description of the height of a divisor and explicit estimates local.

As an application, observe that if $\widehat{h}_{f_\eta}(Y_\eta)\neq 0$, then, for any integer $n$, the set of parameters $ t\in S^0(\mathbb{K})$ where $\mathbb{K}$ is an algebraic extension of $\Q$ with $[\mathbb{K}:\Q]\leq n$ is finite by the Northcott property. Note that the preperiodicity of $Y_t$ implies  
$\widehat{h}_{f_t}(Y_t)=0$  (see e.g.~\cite{zhang}). Recall that an endomorphism $f_t$ of $\mathbb{P}^k$ is post-critically finite (PCF for short) if the critical set is preperiodic, i.e. if there are integers $n>m\geq0$ such that $f^{n}_t(\mathrm{Crit}(f_t))\subset f_t^m(\mathrm{Crit}(f_t))$. Theorem~\ref{tm:principal1} shows that, when $\mathcal{Y}=\mathrm{Crit}(f)$ is the critical set of a family $f$ of endomorphisms of $\mathbb{P}^k$ with $\widehat{h}_{f_\eta}(Y_\eta)\neq 0$ (which means the family is unstable), there are only finitely many post-critically finite (PCF for short) maps on a given extension of $\Q$. 

%

~

Heights can be seen in two different and entangled fashions: by working at all places which can often gives precise estimates and by the mean of arithmetic or algebraic intersection theory which is more intrinsic and allows cohomological arguments. The philosophy of this article is to rely as much as possible on the latter. Our first contribution is a comparison of the naive height and the canonical height in families directly using \cite{CS-height} for sections and using intersection theory for subvarieties of positive relative dimension (see Proposition~\ref{prop:CS-higher}). In a second time, using the exposition \cite{YZ-adelic} of Yuan and Zhang of the \emph{Deligne pairing} \cite{MR902592} of metrized line bundles we deduce the Main Theorem from Proposition~\ref{prop:CS-higher} and from the quasi-equivalence of ample heights on curves.

\subsection*{Acknowledgments}
We would like to thank S\'ebastien Boucksom and Charles Favre for many useful discussions about Deligne pairings.

\section{The canonical height over a number field}\label{sec:canonicalheight}

\subsection{Adelic metrics and their height functions}
Let $X$ be a projective variety of dimension $k$, and let $L_0,\ldots,L_k$ be $\mathbb{Q}$-line bundles on $X$, all defined over a number field $\mathbb{K}$.
Assume $L_i$ is equipped with an adelic continuous metric $\{\|\cdot\|_{v,i}\}_{v\in M_\mathbb{K}}$ and denote $\bar{L}_i:=(L_i,\{\|\cdot\|_v\}_{v\in M_\mathbb{K}})$. Assume $\bar{L}_i$ is  semi-positive for $1\leq i\leq k$. 
 Fix a place $v\in M_\mathbb{K}$. Denote by $X_v^\mathrm{an}$ the Berkovich analytification of $X$ at the place $v$. We also let $c_1(\bar{L}_i)_v$ be the curvature form of the metric $\|\cdot\|_{v,i}$ on $L_v^\mathrm{an}$.

For any closed subvariety $Y$ of dimension $q$ of $X$, the arithmetic intersection number $\left(\bar{L}_0\cdots\bar{L}_q|Y\right)$ is symmetric and multilinear with respect to the $L_i$'s. As observed by Chambert-Loir~\cite{ACL}, we can define $\left(\bar{L}_0\cdots\bar{L}_q|Y\right)$ inductively by
\[\left(\bar{L}_0\cdots\bar{L}_q|Y\right)=\left(\bar{L}_1\cdots\bar{L}_q|\mathrm{div}(s)\cap Y\right)+\sum_{v\in M_\mathbb{K}}n_v\int_{Y_v^{\mathrm{an}}}\log\|s\|^{-1}_v \bigwedge_{j=1}^qc_1(\bar{L}_i)_v,\]
for any global section $s\in H^0(X,L_0)$ such that the intersection $\mathrm{div}(s)\cap Y$ is proper. In particular, if $L_0$ is the trivial bundle and $\|\cdot\|_{v,0}$ is the trivial metric at all places but $v_0$, this gives
\[\left(\bar{L}_0\cdots\bar{L}_q|Y\right)=n_{v_0}\int_{Y_{v_0}^{\mathrm{an}}}\log\|s\|^{-1}_{v_0,0} \bigwedge_{j=1}^qc_1(\bar{L}_i)_{v_0}.\]
When $\bar{L}$ is a big and nef $\mathbb{Q}$-line bundle endowed with a semi-positive continuous adelic metric, following Zhang~\cite{Zhang-positivity}, we can define $h_{\bar{L}}(Y)$ as
\[ h_{\bar{L}}(Y):=\frac{\left(\bar{L}^{q+1}|{Y}\right)}{(q+1)[\mathbb{K}:\mathbb{Q}]\deg_{Y}(L)},\]
where $\deg_Y(L)=(L_{|Y})^q$ is the volume of the line bundle $L$ restricted to $Y$.
\subsection{Canonical height over a number field}
Let $X$ be a projective variety of dimension $k$, let $f:X\to X$ be a morphism and let $L$ be an ample line bundle on $X$, all defined over a number field $\mathbb{K}$. Recall that we say $(X,f,L)$ is a \emph{polarized endomorphism} of degree $d>1$ if $f^*L\simeq L^{\otimes d}$, i.e. $f^*L$ is linearly equivalent to $L^{\otimes d}$. 

It is known that polarized endomorphisms defined over the field $\mathbb{K}$ admit a \emph{canonical metric}. This is an adelic semi-positive continuous metric on $L$, which can be built as follows: let $\mathscr{X}\to\mathrm{Spec}(\mathscr{O}_\mathbb{K})$ be an $\mathscr{O}_\mathbb{K}$-model of $X$ and $\bar{\mathscr{L}}$ be a model of $L$ endowed with a model metric, for example $\bar{\mathscr{L}}=\iota^*\bar{\mathcal{O}}_{\mathbb{P}^N}(1)$, where $\iota: X\hookrightarrow \mathbb{P}^N$ is an embedding inducing $L$ and $\mathcal{O}_{\mathbb{P}^N}(1)$ is endowed with its naive metrization.
We then define $\bar{L}_f$ as
\[\bar{L}_f:=\lim_{n\to\infty}\frac{1}{d^n}(f^n)^*\bar{\mathscr{L}}|_\mathbb{K}.\]
This metrization induces the \emph{canonical height} $\widehat{h}_f$ of $f$: for any closed point $x\in X(\bar{\mathbb{Q}})$ and any section $\sigma\in H^0(X,L)$ which does not vanish at $x$, we let
\[\widehat{h}_{f}(x):=\frac{1}{[\mathbb{K}:\mathbb{Q}]\deg(x)}\sum_{v\in M_\mathbb{K}}\sum_{y\in \mathsf{O}(x)}n_v\log\|\sigma(y)\|_v^{-1},\]
where $x\in X(\mathbb{K})$ , $\mathsf{O}(x)$ is the Galois orbit of $x$ in $X$. The function $\widehat{h}_f:X(\bar{\mathbb{Q}})\to \mathbb{R}$ satisfies $\widehat{h}_f\circ f=d\cdot \widehat{h}_f$, $\widehat{h}_f\geq0$ and $\widehat{h}_f(x)=0$ if and only if $x$ is preperiodic under iteration of $f$, i.e. if there is $n>m\geq0$ such that $f^n(x)=f^m(x)$. Note that $\widehat{h}_f$ can also be defined as
\[\widehat{h}_f(x)=\lim_{n\to\infty}\frac{1}{d^n}h_{X,L}(f^n(x)),\]
where $h_{X,L}$ is any Weil height function on $X$ associated with the ample line bundle $L$.

\section{The canonical height over a function field of characteristic zero}

We now focus on the dynamical setting: let $\pi:\mathcal{X}\to S$ be a family of complex projective varieties, where $S$ is a smooth complex projective curve, and let $\mathcal{L}$ be a nef and relatively ample line bundle on $\mathcal{X}$. We let $f:\mathcal{X}\dashrightarrow\mathcal{X}$ be a rational map such that $(\mathcal{X},f,\mathcal{L})$ is a \emph{family of polarized endomorphisms} of degree $d\geq2$, with regular part $S^0$, i.e. for all $t\in S^0(\C)$, $X_t:=\pi^{-1}\{t\}$ is smooth, $L_t:=\mathcal{L}|_{X_t}$ is ample and $f_t^*L_t\simeq L_t^{\otimes d}$. 

Let $\mathcal{Y}\subsetneq\mathcal{X}$ be a proper subvariety of $\mathcal{X}$ of dimension $q+1$ with $\pi(\mathcal{Y})=S$. Let $S^0_\mathcal{Y}$ be the maximal Zariski open subset of $S^0$ such that the restriction $\pi|_{\mathcal{Y}}:\mathcal{Y}\to S$ of $\pi$ is flat over $S^0$. We denote by $\mathcal{Y}^0$ and $\mathcal{X}^0$ the \emph{regular parts} $\mathcal{Y}^0:=\pi|_{\mathcal{Y}}^{-1}(S^0_\mathcal{Y})$ and $\mathcal{X}^0 :=\pi^{-1}(S^0_\mathcal{Y})$.

Let $\omega$ be a smooth positive form representing the first Chern class $c_1(\mathcal{L})$ on $\mathcal{X}$. As $f^*\mathcal{L}\simeq \mathcal{L}^{\otimes d}$ on $\mathcal{X}^0$, there is a smooth function $g:\mathcal{X}^0\to\mathbb{R}$ such that $d^{-1}f^*\omega=\omega+dd^cg$ as forms on $\mathcal{X}^0$. In particular, the following limit exists as a closed positive $(1,1)$-current on $\mathcal{X}^0$
\[ \widehat{T}_f := \lim_{n\to\infty} \frac{1}{d^n} (f^n)^*(\omega),\]
and can be written as $\widehat{T}_f=\omega+dd^cg_f$, where $g_f:= \sum_{n=0}^\infty d^{-n} g\circ f^n$ is continuous on $\mathcal{X}^0$. The 
current $\widehat{T}_f$ is the \emph{fibered Green current} of $f$.

\medskip


Let $Y_\eta$ be the generic fiber of a family $\mathcal{Y}\to S$ of subvarieties of relative dimension $q$ of $\mathcal{X}\to S$, and let $\phi_n:\mathcal{X}_n\to\mathcal{X}$ be a birational morphism such that $f^n\circ \phi_n$ extends as a morphism $F_n:\mathcal{X}_n\to\mathcal{X}$. We define
\[\widehat{h}_{f_\eta}(Y_\eta):=\lim_{n\to\infty}d^{-n(q+1)}\frac{\left((F_n)_*\phi_n^*\{\mathcal{Y}\}\cdot c_1(\mathcal{L})^{q+1}\right)}{(q+1)\deg_{Y_\eta}(L_\eta)}.\]

The next lemma follows from~\cite{GV_Northcott}:
\begin{lemma}\label{lm:GV}
For any $\mathcal{Y}$ as above, $\widehat{h}_{f_\eta}(Y_\eta)$ is well-defined and satisfies $\widehat{h}_{f_\eta}((f_\eta)_*(Y_\eta))=d\widehat{h}_{f_\eta}(Y_\eta)$. In addition, we can compute $\widehat{h}_{f_\eta}(Y_\eta)$ as
\[\widehat{h}_{f_\eta}(Y_\eta)=\frac{1}{(q+1)\deg_{Y_\eta}(L_\eta)}\int_{\mathcal{X}^0(\C)}\widehat{T}^{q+1}_f\wedge[\mathcal{Y}].\]
\end{lemma}

\begin{proof}
The fact that it is well-defined and the formula relating the limit of $d^{-n(q+1)}\left((f^n)_*\{\mathcal{Y}\}\cdot c_1(\mathcal{L})^{q+1}\right)$ with $\widehat{T}_f^{q+1}\wedge\left(f_*[\mathcal{Y}]\right)$ are contained in \cite[Theorem~B]{GV_Northcott}. We then can compute
\begin{align*}
\widehat{h}_{f_\eta}((f_\eta)_*(Y_\eta)) & =\frac{1}{(q+1)\deg_{Y_\eta}(f_\eta^*L_\eta)}\int_{\mathcal{X}^0(\C)}\widehat{T}_f^{q+1}\wedge\left(f_*[\mathcal{Y}]\right)\\
& = \frac{1}{(q+1)d^q\deg_{Y_\eta}(L_\eta)}\int_{\mathcal{X}^0(\C)}\left(f^*\widehat{T}_f^{q+1}\right)\wedge [\mathcal{Y}]\\
& = \frac{d^{q+1}}{(q+1)d^q\deg_{Y_\eta}(L_\eta)}\int_{\mathcal{X}^0(\C)}\widehat{T}_f^{q+1}\wedge [\mathcal{Y}]=d\widehat{h}_{f_\eta}(Y_\eta),
\end{align*}
where we used that $f^*(\widehat{T}_f)= d  \widehat{T}_f$, $\dim Y_\eta=q$, and $\dim \mathcal{Y}=q+1$.
\end{proof}

 In particular, the last part of the lemma states that the height $\widehat{h}_{f_\eta}(Y_\eta)$ is $>0$ if and only if the measure $\widehat{T}^{q+1}_f\wedge[\mathcal{Y}]$  is not identically zero on $\mathcal{X}^0(\C)$.

 Let $\pi_n:= \pi \circ \phi_n :\mathcal{X}_n\to S$. Relying on estimates from \cite{GV_Northcott} we can deduce

\begin{lemma}\label{cor:GV}
There is a constant $C\geq1$ depending only on $(\mathcal{X},f,\mathcal{L})$ and $\mathcal{Y}$ such that for any ample $\mathbb{Q}$-line bundle $\mathcal{M}$ on $S$ of degree $1$ and any $n\geq1$, we have
\[\left|\frac{\left(\phi_n^*\{\mathcal{Y}\}\cdot (F_n)^*c_1(\mathcal{L})^{q+1}\right)}{(q+1)\left(\phi_n^*\{\mathcal{Y}\}\cdot (F_n)^*c_1(\mathcal{L})^{q}\cdot c_1(\pi_n^*\mathcal{M})\right)}-d^n\widehat{h}_{f_\eta}(Y_\eta)\right|\leq C.\]
\end{lemma}

\begin{proof}
Combining Proposition~3.5 and Theorem~B from \cite{GV_Northcott} we have
\[d^{-n(q+1)}\left(\phi_n^*\{\mathcal{Y}\}\cdot (F_n)^*c_1(\mathcal{L})^{q+1}\right)=\int_{\mathcal{X}^0(\C)}\widehat{T}^{q+1}_f\wedge[\mathcal{Y}]+O\left(\frac{1}{d^n}\right).\]
Let now $\alpha$ be a smooth form on $S(\C)$ which represents $c_1(\mathcal{M})$ (it has mass $1=\deg_S(\mathcal{M}) $) and $\omega$ be a smooth form on $\mathcal{X}(\C)$ which represents $c_1(\mathcal{L})$. By definition, we have
\begin{align*}
\left(\phi_n^*\{\mathcal{Y}\}\cdot (F_n)^*c_1(\mathcal{L})^{q}\cdot c_1(\pi_n^*\mathcal{M})\right) & =\int_{\mathcal{X}^0(\C)}\left((f^n)^*\omega\right)^q\wedge[\mathcal{Y}]\wedge\pi^*(\alpha)\\
& = \int_{S^0_\mathcal{Y}(\C)}\pi_*\left(\left((f^n)^*\omega\right)^q\wedge[\mathcal{Y}]\right)\wedge \alpha\\
&= \int_{S^0_\mathcal{Y}(\C)}\left(\int_{Y_t}\left((f_t^n)^*\omega_t\right)^q\right)\wedge \alpha\\
&= d^{qn}\int_{S^0_\mathcal{Y}(\C)}\left(\deg_{Y_t}(L_t)\right)\alpha=d^{qn}\deg_{Y_\eta}(L_\eta),
\end{align*}
where we used that $\dim Y_t=q$, $\dim \mathcal{Y}^0=q+1$ and that $\left((f^n)^*\omega\right)^q\wedge[\mathcal{Y}]$ has bidegree $(q,q)$ on $\mathcal{Y}^0(\C)$ so that $\pi_*\left(\left((f^n)^*\omega\right)^q\wedge[\mathcal{Y}]\right)$ has bidegree $(0,0)$ on $S^0_\mathcal{Y}(\C)$, i.e. is a function, since the fibers of $\pi$ have dimension $q$.
\end{proof}

\section{Comparing the canonical and the naive heights in families}

As above, let $(\mathcal{X},f,\mathcal{L})$ be a family of polarized endomorphisms of degree $d\geq2$ defined over $\mathbb{K}$, with regular part $S^0$. We endow $\mathcal{L}$ with a semi-positive adelic continuous metrization $\bar{\mathcal{L}}$. We let $\mathcal{Y}\subsetneq\mathcal{X}$ be a subvariety defined over $\mathbb{K}$ and such that $\pi(\mathcal{Y})=S$, and let $S^0_\mathcal{Y}$ be the maximal Zariski open subset of $S^0$  such that $\pi|_{\mathcal{Y}}$ is flat over $S^0_\mathcal{Y}$. We also endow $S$ with an ample divisor $H$ of degree $1$.


We prove here the following higher dimensional counterpart to Call and Silverman's pointwise estimate~\cite[Theorem~3.1]{CS-height}, see~\cite[Theorem~1]{Ingram-variation} for the case of hypersurfaces of $\mathbb{P}^k$
\begin{proposition}\label{prop:CS-higher}
There exists a constant $C\geq1$ depending only on the family $(\mathcal{X},f,\mathcal{L})$ and the heights $h_{\bar{\mathcal{L}}}$ and $h_{S,H}$ such that for any subvariety $\mathcal{Y}\subsetneq \mathcal{X}$ such that $\left(\mathcal{X},f,\mathcal{L},\mathcal{Y}\right)$ is a dynamical pair with regular part $S^0_\mathcal{Y}$ and for any $t\in S^0_\mathcal{Y}(\bar{\mathbb{Q}})$ we have
\[\left|h_{\bar{\mathcal{L}}}(Y_t)-\widehat{h}_{f_t}(Y_t)\right|\leq C\left(h_{S,H}(t)+1\right).\]
\end{proposition}

\begin{proof}
Let $q$ be the relative dimension of $\mathcal{Y}$ and $\mathbb{K}$ be a finite extension of $\mathbb{Q}$ over which $\mathcal{Y}$ and  $t$ are defined. We let $\mathcal{D}$ be a divisor of $\mathcal{X}$ which represents $\mathcal{L}$ and we decompose the height functions $h_{\bar{L}}$ and $\widehat{h}_{f_t}$ using this representative of $\mathcal{L}$:
\[h_{\bar{\mathcal{L}}}=\frac{1}{[\mathbb{K}:\mathbb{Q}]}\sum_{v\in M_\mathbb{K}}n_v\lambda_{\mathcal{D},v} \quad \text{and} \quad \widehat{h}_{f_t}=\frac{1}{[\mathbb{K}:\mathbb{Q}]}\sum_{v\in M_\mathbb{K}}n_v\widehat{\lambda}_{f_t,D_t,v},\]
where $\widehat{\lambda}_{f_t,D_t,v}\circ f_t=d\cdot \widehat{\lambda}_{f_t,D_t,v}$ and $\widehat{\lambda}_{f_t,D_t,v}=\lambda_{\mathcal{D},v}|_{X_t}+O_v(1)$, where $O_v(1)=0$ for all but finitely places $v\in M_\mathbb{K}$. We also let $h_{S,H}=\frac{1}{[\mathbb{K}:\mathbb{Q}]}\sum_{v\in M_\mathbb{K}}n_v\lambda_{H,v}$.

We rely on a key estimate of Call and Silverman~\cite[Theorem~3.2]{CS-height}: there is a constant $C_1\geq1$ depending only on the family $(\mathcal{X},f,\mathcal{L})$, and the heights $h_{\bar{\mathcal{L}}}$ and $h_{S,H}$ such that for any $t\in S^0(\bar{\mathbb{Q}})$, any $x\in X_t(\bar{\mathbb{Q}})\setminus\mathrm{supp}(D_t)$, and any $v\in M_\mathbb{K}$, we have
\begin{align}
\left|\lambda_{\mathcal{D},v}(x)-\widehat{\lambda}_{f_t,D_t,v}(x)\right|\leq C(v)(\lambda_{H,v}(t)+1).\label{ineg:CS-local}
\end{align}
 with $C(v)=C_1\geq1$ for all $v$ in a finite set $S\subset M_\mathbb{K}$ containing all archimedean places, and $C(v)=0$ otherwise.
Moreover, the constant $C_1$ depends only on the choice of $\mathcal{D}$ and on the choice of the above decompositions.

We now fix $t\in S^0_\mathcal{Y}(\bar{\mathbb{Q}})$ and let $q:=\dim Y_t$ (which is independent of $t\in S_\mathcal{Y}^0(\bar{\mathbb{Q}})$). By definition, we have
\begin{align*}
h_{\bar{\mathcal{L}}}(Y_t)-\widehat{h}_{f_t}(Y_t) & = \frac{1}{(q+1)[\mathbb{K}:\mathbb{Q}]\deg_{Y_t}(L_t)}\sum_{v\in M_\mathbb{K}}n_v\left(\left(\bar{L}_t^{q+1}| Y_t\right)_v-\left(\bar{L}_{t,f_t}^{q+1}| Y_t\right)_v\right)
\end{align*}
Fix now a place $v\in M_\mathbb{K}$. Then we can compute
\begin{align*}
\left(\bar{L}_t^{q+1}| Y_t\right)_v-\left(\bar{L}_{t,f_t}^{q+1}| Y_t\right)_v & =\sum_{j=0}^q \left(\left(\bar{L}_t-\bar{L}_{t,f_t}\right)\cdot \bar{L}_t^{j}\cdot \bar{L}_{t,f_t}^{q-j}|Y_t\right)_v\\
& = \sum_{j=0}^q\int_{Y_{t,v}^\mathrm{an}}\log\|1\|_{t,v}^{-1}\cdot c_1(\bar{L}_t)_v^j\wedge  c_1(\bar{L}_{t,f})_v^{q-j}\\
& = \sum_{j=0}^q\int_{Y_{t,v}^\mathrm{an}}\left(\lambda_{\bar{\mathcal{L}_t},v}-\widehat{\lambda}_{f_t,D_t,v}\right)\cdot c_1(\bar{L}_t)_v^j\wedge  c_1(\bar{L}_{t,f})_v^{q-j},
\end{align*}
where we used that the local height function $\lambda_{\bar{\mathcal{L}}_t,v}-\widehat{\lambda}_{f_t,D_t,v}$ extends as a continuous metric on the trivial bundle, since $h_{\bar{L}_t}$ and $\widehat{h}_{f_t}$ are induced by adelic continuous metrization on the same line bundle $L_t$. Combined with \eqref{ineg:CS-local}, this gives
\begin{align*}
\left|\left(\bar{L}_t^{q+1}| Y_t\right)_v-\left(\bar{L}_{t,f_t}^{q+1}| Y_t\right)_v\right| & \leq C(v)\left(\lambda_{S,H,v}(t)+1\right) \sum_{j=0}^q\int_{Y_{t,v}^\mathrm{an}} c_1(\bar{L}_t)_v^j\wedge  c_1(\bar{L}_{t,f})_v^{q-j}\\
& \leq C(v)\left(\lambda_{S,H,v}(t)+1\right) \sum_{j=0}^q\left(L_t^j\cdot L_t^{q-j}\cdot Y_t\right)\\
& \leq C(v)\left(\lambda_{S,H,v}(t)+1\right) (q+1)\deg_{Y_t}(L_t),
\end{align*}
since the measures $c_1(\bar{L}_t)_v^j\wedge  c_1(\bar{L}_{t,f})_v^{q-j}$ don't give mass to the closed subvariety $D_t\cap Y_t$, seen as a pluripolar subset of $X_{t,v}^{\mathrm{an}}$, see e.g. \cite[Lemma~8.6]{Boucksom-Eriksson} for non-archimedean $v\in M_\mathbb{K}$.
As we have $h_{S,H}=\frac{1}{[\mathbb{K}:\mathbb{Q}]}\sum_{v\in M_\mathbb{K}}n_v\lambda_{S,H,v}$, summing over all places and dividing by $(q+1)[\mathbb{K}:\mathbb{Q}]\deg_{Y_t}(L_t)$ gives
\[\left|h_{\bar{\mathcal{L}}}(Y_t)-\widehat{h}_{f_t}(Y_t)\right|\leq C_1\left(h_{S,H}(t)+1\right), \]
for all $t\in S^0_\mathcal{Y}(\bar{\mathbb{Q}})$, which is the wanted estimate, $\mathrm{supp}(D_t)\cap \mathrm{supp}(Y_t)$  is not a component of $\mathrm{supp}(Y_t)$.

\medskip

Let us now replace $\mathcal{D}$ by another divisor representing $\mathcal{L}$ in a finite family of such divisors so that we can make sure that for any family $\mathcal{Y}\to S$ and any $t\in S^0_\mathcal{Y}(\bar{\mathbb{Q}})$, there is a choice $\mathcal{D}^{(i)}$ such that $\mathrm{supp}(D^{(i)}_t)\cap \mathrm{supp}(Y_t)$  is not a component of $\mathrm{supp}(Y_t)$. Replacing $C_1$ by $\max_iC_1(\mathcal{D}^{(i)})$ gives the wanted estimate.
\end{proof}

\section{Variation of canonical heights of subvarieties}\label{sec:Var2}

\subsection{Variation of naive heights of subvarieties}

The material here follows the presentation of Yuan and Zhang~\cite{YZ-adelic} of the Deligne pairing (\cite{MR902592}).
Let $S$ be a smooth and integral projective curve defined over a number field $\mathbb{K}$.
Let $\pi:\mathcal{X}\to S$ be a projective and flat morphism defined over $\mathbb{K}$. Let $D:=\dim(\mathcal{X})-1>0$ be its relative dimension.
Let $\bar{L}$ be a model ample line bundle on $\mathcal{X}$, i.e. there is a $\mathscr{O}_\mathbb{K}$-model $\mathscr{X}$ of $X$, together with an hermitian line bundle $\bar{\mathscr{L}}$ which restricts as $\bar{L}$ on the generic fiber of the structure morphism $\mathscr{X}\to\mathrm{Spec}(\mathscr{O}_\mathbb{K})$. One can define an adelic metrized ample line bundle on $S$ as the \emph{Deligne pairing} $\langle \bar{L}^{D+1}\rangle$. By \cite{YZ-adelic}, we can easily prove the following

\begin{theorem}\label{tm:Deligne-pairing}
Let $S$ be a smooth integral projective curve and $\mathcal{X}$ be an integral projective variety, both defined over a number field. Assume there is a flat and projective morphism $\pi:\mathcal{X}\to S$ of relative dimension $D$, also defined over a number field. Let $\bar{L}$ be a big and nef line bundle on $\mathcal{X}$, equipped with a model metric.

Then $\bar{M}:=\left((D+1)\deg_{X_\eta}(L_\eta)\right)^{-1}\langle\bar{L}^{D+1}\rangle$ is an adelic semi-positive continuous ample line bundle on $S$ whose induced height function is given by
\[h_{\bar{M}}(t)=h_{\bar{L}}(X_t), \quad t\in S(\bar{\mathbb{Q}}).\]
Moreover, for any place $v\in M_\mathbb{K}$, the measure $c_1(\bar{M})_v$ is $\pi_*c_1(\bar{L})_v^D$ and $\deg_{S}(M)=h_{L_\eta}(X_\eta)$, where $X_\eta$ is the generic fiber of $\pi$ and $L_\eta$ is the restrictions of $L$ to $X_\eta$.
\end{theorem}

\begin{proof}
Fix a $\mathscr{O}_\mathbb{K}$-model $\pi:\mathscr{X}\to\mathscr{S}$ of $\pi:X\to S$ which is flat and projective and which induces the hermitian line bundle $\bar{L}$. Yuan and Zhang \cite[\S 4.4]{YZ-adelic} prove that $\bar{\mathscr{M}}:=\langle \bar{\mathscr{L}}^{D+1}\rangle$ is an ample hermitian line bundle on $\mathscr{S}$ and that one can compute
\begin{align*}
h_{\langle\bar{L}^{D+1}\rangle}(t) = \frac{\left(\langle\bar{\mathscr{L}}^{D+1}\rangle|_{\bar{t}}\right)}{\deg(\bar{t})}=\frac{\left(\bar{\mathscr{L}}^{D+1}|_{\bar{X}_t}\right)}{\deg(\bar{t})},
\end{align*}
where $\bar{t}$ (resp. $\bar{X}_t$) is the closure of $t$ (resp. of $X_t$) in the scheme $\mathscr{X}$.
Note that the last quantity is precisely $(D+1)\deg_{X_{t}}(L_t)h_{\bar{L}}(X_t)$. As $\pi$ is projective and flat, $\deg_{X_{t}}(L_t)=\deg_{X_{\eta}}(L_\eta)$ for all $t$. We deduce the wanted properties of $\bar{M}$ noticing that $\bar{M}$ is the restriction of $\bar{\mathscr{M}}$ to the special fiber of the structure morphism $\mathscr{S}\to\mathrm{Spec}(\mathscr{O}_\mathbb{K})$.

All there is left to do is to compute the measure $c_1(\bar{M})_v$ at an archimedean place $v\in M_\mathbb{K}$. This is done in \cite[\S 4.3.2]{YZ-adelic} where $c_1(\bar{M})_v=\pi_*(c_1(\bar{L})_v^D)$ is proved, which concludes the proof. 
\end{proof}

\subsection{From comparison of heights to variation of heights}

We now come back to the dynamical setting: let $\left(\mathcal{X},f,\mathcal{L},\mathcal{Y}\right)$ be a dynamical pair parametrized by a smooth projective curve $S$, all defined over a number field $\mathbb{K}$, with regular part $S^0$.

~


In what follow, we say that  the dynamical pair $\left(\mathcal{X},f,\mathcal{L},\mathcal{Y}\right)$ is \emph{unstable} if $\widehat{h}_{f_\eta}(Y_\eta)\neq 0$. We now prove the following, which implies the main theorem.

\begin{theorem}\label{tm:main}
Let $\left(\mathcal{X},f,\mathcal{L},\mathcal{Y}\right)$ be a dynamical pair parametrized by $S$ with regular part $S^0$, all defined over a number field $\mathbb{K}$. For any $\mathbb{Q}$-ample height $h_S$ on $S$ of degree $1$ and any $\varepsilon>0$, there exists a constant $C(\varepsilon)>0$ such that, the following holds for all $t\in S_0(\bar{\mathbb{Q}})$,
\[\left(\widehat{h}_{f_\eta}(Y_\eta)-\varepsilon\right)h_S(t)-C(\varepsilon)\leq\widehat{h}_{f_t}(Y_t)\leq \left(\widehat{h}_{f_\eta}(Y_\eta)+\varepsilon\right)h_S(t)+C(\varepsilon).\]
In particular, if the dynamical pair $\left(\mathcal{X},f,\mathcal{L},\mathcal{Y}\right)$ is unstable, the function $t\mapsto \widehat{h}_{f_t}(Y_t)$ is an ample height on $S$.
\end{theorem}

\begin{proof}
As $f$ is a finite endomorphism on $\mathcal{X}^0$ and $S^0_{f^n(\mathcal{Y})}=S^0_{\mathcal{Y}}$ for any $n\geq1$, we can apply Proposition~\ref{prop:CS-higher} to the cycle $(f_t^n)_*(Y_t)$ for all $t\in S^0_\mathcal{Y}(\bar{\mathbb{Q}})$. This is possible since $(f_t^n)_*(Y_t)=\deg(f_t^n|_{Y_t})\cdot f_t^n(Y_t)$ and $f_t^n(Y_t)$ is irreducible at least when $Y_t$ is.
\[\left|h_{\bar{\mathcal{L}}}((f_t^n)_*(Y_t))-\widehat{h}_{f_t}((f_t^n)_*(Y_t))\right|\leq C\left(h_{S,H}(t)+1\right).\]
Let now $\phi_n:\mathcal{X}_n\to\mathcal{X}$ be a birational morphism such that there is a morphism $F_n:\mathcal{X}_n\to\mathcal{X}$ with $F_n=f^n\circ\phi_n$ on $\phi_n^{-1}(\mathcal{X}^0)$ and let $\bar{\mathcal{L}}_n:=\left(d^{-n}F_n\right)^*\bar{\mathcal{L}}$. As $F_n$ is a generically finite morphism and $\bar{\mathcal{L}}$ is an ample adelic semi-positive continuous metrized line bundle, the line bundle $\bar{\mathcal{L}}_n$ is an adelic semi-positive continuous metrized big and nef line bundle on $\mathcal{X}_n$.
Set now $\mathcal{Y}_n:=\phi_n^{-1}(\mathcal{Y})$. Up to applying the Raynaud-Gruson flattening theorem~\cite[Theorem~5.2.2]{RG}, we can assume $\mathcal{Y}_n\to S$ is flat and projective. Now, we define a hermitian line bundle $\bar{L}_n$ on $\mathcal{Y}_n$ by restricting $\bar{\mathcal{L}}_n$ to $\mathcal{Y}_n$.
Since for any $t\in S^0(\bar{\mathbb{Q}})$, we have $h_{\bar{\mathcal{L}}}((f_t^n)_*(Y_t))=h_{(f^n)^*\bar{\mathcal{L}}}(Y_t)=h_{\bar{L}_n}(\phi_n^{-1}(Y_t))$, by the invariance property $\widehat{h}_{f_t}((f_t)_*(Y_t))=d\widehat{h}_{f_t}(Y_t)$, this gives
\begin{align}
\left|h_{\bar{L}_n}(\phi_n^{-1}(Y_t))-\widehat{h}_{f_t}(Y_t)\right|\leq \frac{C}{d^n}\left(h_{S,H}(t)+1\right).\label{ineq:at-time-n}
\end{align}
We now rely on Theorem~\ref{tm:Deligne-pairing}: the function $t\mapsto h_{\bar{L}_n}(\phi_n^{-1}(Y_t))$ is a Weil height function associated with an ample adelic semi-positive continuous $\mathbb{Q}$-line bundle $M_n$ on $S$.
 Moreover, the degree of this line bundle is given by
\begin{align*}
\deg(M_n) & =\frac{1}{d^n}h_{L_\eta}((f^n)_*(Y_\eta))\\
&=\frac{1}{d^n(q+1)\mathrm{vol}(((f_\eta^n)^*L_\eta)|_{Y_\eta})}\left(c_1(\mathcal{L})^{q+1}\cdot (F_n)_*\{\mathcal{Y}_n\}\right)\\
& =\frac{1}{(q+1)d^{n(q+1)}\mathrm{vol}(L_\eta|_{Y_\eta})}\left(c_1(\mathcal{L})^{q+1}\cdot (F_n)_*\{\mathcal{Y}\}\right)\\
& = \widehat{h}_{f_\eta}(Y_\eta)+O(d^{-n}),
\end{align*}
where we used Lemma~\ref{lm:GV}.
We now use the quasi-equivalence of ample height functions on a projective curve, see~e.g.~\cite[Chapter~4, Corollary~3.5]{Lang-diophantine}: for any two height functions $h_1,h_2$ induced by two ample line bundles $L_1,L_2$ on $S$ respectively, then
\[\lim_{h_1(t)\to\infty}\frac{h_2(t)}{h_1(t)}=\frac{\deg(L_2)}{\deg(L_1)}.\]
Fix now any ample height $h_S$ on $S$ induced by an ample $\mathbb{Q}$-line bundle of degree $1$. We deduce from the above that $h_{\bar{\mathcal{L}}}((f_t^n)_*(Y_t))=\left(d^n\widehat{h}_{f_\eta}(Y_\eta)+O(1)\right)h_S(t)+\varepsilon_n(h_S(t))$, where $\varepsilon_n(h_S(t))=o(h_S(t))$ depends on $n$. Together with \eqref{ineq:at-time-n}, this gives
\[\left|\widehat{h}_{f_\eta}(Y_\eta)h_S(t)-\widehat{h}_{f_t}(Y_t)\right|\leq \frac{C_1}{d^n}\left(h_S(t)+h_{S,H}(t)+1\right)+\varepsilon_n(h_S(t)),\]
for all $t\in S^0(\bar{\mathbb{Q}})$. Again by quasi-equivalence of ample heights, we have $h_{S,H}\leq C_2 (h_S+1)$ since $H$ is ample and $h_S$ is induced by an ample line bundle, where $C''$ depends only on $\deg(H)$. Fix $n>1$ large enough so that $2C_1(1+C_2)\leq d^n\varepsilon$. We then have
\[\left|\widehat{h}_{f_\eta}(Y_\eta)h_S(t)-\widehat{h}_{f_t}(Y_t)\right|\leq \frac{\varepsilon}{2} h_S(t)+h_{S,H}(t)+C_3+\varepsilon_n(h_S(t)),\]
for all $t\in S^0(\bar{\mathbb{Q}})$, where $C_3>0$ is a constant depending on $\varepsilon>0$. Now, as $\varepsilon_n(h_S(t))=o(h_S(t))$, there exists $B(\varepsilon)\geq1$ such that if $h_S(t)\geq B(\varepsilon)$, then $\varepsilon_n(h_S(t))\leq \varepsilon h_S(t)/2$ and we have $\varepsilon_n(h_S(t))\leq B(\varepsilon)+\frac{\varepsilon}{2}h_S(t)$. The conclusion follows letting $C(\varepsilon):=C_3+(\varepsilon)$.
\end{proof}

An immediate consequence is the Theorem from the introduction:

\begin{proof}[Proof of the Main Theorem]
Fix $\varepsilon>0$, divide the inequalities obtained in Theorem~\ref{tm:main} by $h_S(t)$ and make it tend to $\infty$ to find
\[\left|\lim_{h_S(t)\to\infty \atop t\in S^0(\bar{\mathbb{Q}})}\frac{\widehat{h}_{f_t}(Y_t)}{h_S(t)}-\widehat{h}_{f_\eta}(Y_\eta)\right|\leq \varepsilon.\]
As this holds for any $\varepsilon>0$, the result follows.
\end{proof}

\bibliographystyle{short}
\bibliography{biblio}
\end{document}